\documentclass[a4paper,12pt]{amsart} 
\usepackage{amssymb, amsmath, amsthm, color, amstext, tikz, mathrsfs}
\usepackage[colorlinks=true, breaklinks=true, linkcolor=red, citecolor=blue, urlcolor=blue]{hyperref} 
\usepackage[english]{babel}
\numberwithin{equation}{section}
\newtheorem{remark}{Remark}
\newtheorem{lemma}{Lemma}

\newcommand{\forestftwo}{
\begin{tikzpicture}[baseline=.4cm]
\draw (0,0)--(0,2/3);
\draw (0,2/3)--(-1/3,1);
\draw (0,2/3)--(1/3,1);
\draw (-2/3,0)--(-2/3,1);
\draw (2/3,0)--(2/3,1);
\draw (1,0)--(1,1);
\end{tikzpicture}
}
\newcommand{\fonefone}{
\begin{tikzpicture}[baseline=.8cm]
\draw (0,0)--(0,-.5);
\draw (0,0)--(-1,2);
\draw (-.5,1)--(0,2);
\draw (0,0)--(1,2);
\node at (1,1.5) {$k$};
\node at (.1, 1.5) {$m$};
\node at (-.9 , 1.5) {$l$};
\node at (-.5,.5) {$j$};
\node at (-.2,-.2) {$i$};
\end{tikzpicture}
}

\newcommand{\cAF}{\mathcal{AF}}
\newcommand{\al}{\alpha}

\newcommand{\cC}{\mathcal C}
\newcommand{\cD}{\mathcal D}
\newcommand{\de}{\delta}
\newcommand{\cF}{\mathcal F}
\newcommand{\bF}{\mathbb F}
\newcommand{\fH}{\mathfrak H}
\newcommand{\scrH}{\mathscr H}
\DeclareMathOperator{\Hilb}{Hilb}
\DeclareMathOperator{\id}{id}
\newcommand{\N}{\mathbf{N}}
\newcommand{\ot}{\otimes}
\newcommand{\pia}{\pi_\alpha}

\newcommand{\R}{\mathbf R}
\newcommand{\cSF}{\mathcal{SF}}
\DeclareMathOperator{\target}{target}
\newcommand{\un}{\underline}
\newcommand{\cU}{\mathcal U}
\newcommand{\vara}{\varphi_\alpha}
\newcommand{\Z}{\mathbf{Z}}

\begin{document}

\title[Haagerup and Kazhdan properties]{On the Haagerup and Kazhdan properties of R. Thompson$'$s groups.}
\thanks{A.B. is supported by European Research Council Advanced Grant 669240 QUEST.
V.J. is supported by the grant numbered DP140100732, Symmetries of subfactors}
\author{Arnaud Brothier and Vaughan F. R. Jones}
\address{Arnaud Brothier\\Department of Mathematics, University of Rome Tor Vergata, Via della Ricerca Scientifica 00133 Roma, Italy}\address{School of Mathematics and Statistics, University of New South Wales, Sydney NSW 2052, Australia}
\email{arnaud.brothier@gmail.com\endgraf
\url{https://sites.google.com/site/arnaudbrothier/}}
\address{Vaughan F. R. Jones\\ Vanderbilt University, Department of Mathematics, 1326 Stevenson Center Nashville, TN, 37240, USA}\email{vaughan.f.jones@vanderbilt.edu}

\begin{abstract}
A machine developed by the second author produces a rich family of unitary representations of the Thompson groups F,T and V.
We use it to give direct proofs of two previously known results.
First, we exhibit a unitary representation of $V$ that has an almost invariant vector but no nonzero $[F,F]$-invariant vectors reproving, at least for $T$, Reznikov's result that any intermediate subgroup between the commutator subgroup of $F$ and $V$ does not have Kazhdan's property (T).
Second, we construct a one parameter family interpolating between the trivial and the left regular representations of $V$.
We exhibit a net of  coefficients for those representations which vanish at infinity on $T$ and converge to $1$ thus reproving Farley's
result that $T$ has the Haagerup property.
\end{abstract}

\maketitle
\vspace{-.4cm}

\section*{Introduction}
Let $F\subset T\subset V$ be the usual Thompson's groups acting on the the unit interval, see \cite{Cannon-Floyd-Parry96}.
In \cite{Jones16-Thompson} it was shown that certain categories $\cC$ with a privileged object $1\in \cC$ give rise to a group of fractions $G_\cC$ and that a functor $\Phi:\cC\to\cD$ provides an action of the group $G_\cC$.
Similar ideas in the context of semigroups were developed by Ore, see for instance  \cite{Maltsev53}.
Thompson's groups $F,T,V$ can be constructed in this way using various categories of forests written $\cF,\cAF,\cSF$.
If $\cD=\Hilb$ is the category of Hilbert spaces with isometries for morphisms, then the functor $\Phi$ gives us a unitary representation of $G_\cC$.
In this article, we consider functors $\Phi:\cF\to\Hilb$ such that $\Phi(n) = \fH^{\ot n}$ and $\Phi(f_{i,n})=\id^{\ot i-1}\ot R\ot \id^{\ot n-i}$ where $R:\fH\to\fH\ot \fH$ is a fixed isometry and $f_{i,n}$ is the forest with $n$ trees all of which are trivial except the $i$th one that has two leaves.
Since any tree is a composition of $f_{i,n}$ we obtain a well defined functor and thus a unitary representation of $F$ that we extend to $V$ via permutations of the tensors, see Section \ref{sec:def} for more details.

In Section \ref{sec:T}, we construct a unitary representation of $V$ having an almost invariant vector but no nonzero $[F,F]$-invariant vectors showing that any intermediate subgroups between $[F,F]$ and $V$ does not have Kazhdan's property (T) \cite{Kazhdan67-T}.
Note that $T$ was proved to not have property (T) initially by Reznikov \cite{Reznikov01} as well as follows from the work of Ghys-Sergiescu and Navas \cite{Ghys-Sergiescu87,Navas02}.

In Section \ref{sec:H}, we give a one parameter family $(\pi_\al, 0\leq \al\leq 1)$ of unitary representations of $V$ interpolating the trivial and the left regular one.
We define some positive definite maps $\vara$ from those representations for which we can explicitly compute their values on $T$.
They actually coincide on $T$ with the family of maps constructed from the proper cocycle of Farley, see Remark \ref{rem:Farley}, but differ on the larger group $V$.
Hence, they provide a $C_0$ positive definite approximation of the identity for $T$ and thus proving Farley's result that $T$ has the Haagerup property (though Farley actually proved it for all of $V$)\cite{Akemann-Walter81, Farley03-H}.

\subsection*{Acknowledgement}
We thank the generous support of the New Zealand Mathematics Research Institute and the warm hospitality we received in Raglan which made  this work possible.

\section{Definitions and notations}\label{sec:def}
We briefly recall the construction of actions of  groups of fractions for the particular cases of $F,T,$ and $V$ and refer to \cite{Jones16-Thompson} and \cite{Cannon-Floyd-Parry96} for more details.
Let $\cF$ be the category of (binary planar) forests whose objects are the natural numbers $\N:=\{1,2,\cdots\}$ and morphisms $\cF(n,m)$ the set of binary planar forests with $n$ roots and $m$ leaves. 
We think of them as planar diagram in the plane $\R^2$ whose roots and leaves are distinct points in $\R\times\{0\}$ and $\R\times \{1\}$ respectively and are counted from left to right.
We compose forests by stacking them vertically so that $p\circ q$ is the forest obtained by stacking on top of $q$ the forest $p$ where the $i$th root of $p$ is attached to the $i$th leaf of $q$. 
We obtain a diagram in the strip $\R\times [0,2]$ that we rescale in $\R\times [0,1].$
For any $n\in \N, 1\leq i\leq n$ we consider the forest $f_{i,n}$, or simply $f_i$ if the context is clear, the forest with $n$ roots and $n+1$ leaves where the $i$th tree of $f_{i,n}$ has two leaves.
For example $$f_{2,4}=\forestftwo\ .$$
Consider the set of pairs of \emph{trees} $(t,s)$ with the same number of leaves, that we quotient by the relation generated by $(t,s)\sim (p\circ t , p\circ s)$ for any forest $p$.
We write $\frac{t}{s}$ for the equivalence class of $(t,s)$.
These form the group of fractions of the category $\cF$ with the multiplication $\frac{t}{s}\cdot \frac{s}{r} = \frac{t}{r}$ and inverse $(\frac{t}{s})^{-1}=\frac{s}{t}$.
It is isomorphic to Thompson's group $F$.

Now consider the category of symmetric forests $\cSF$ with objects $\N$ and morphisms $\cSF(n,m)=\cF(n,m)\times S_m$ where $S_m$ is the symmetric group of $m$ elements. 
Graphically we interpret a morphism $(p,\tau)\in\cSF(n,m)$ as the concatenation of two diagrams.
On the bottom we have the diagram explained above for the forest $p$ in the strip $\R\times [0,1].$
The diagram of $\tau$ is the union of $m$ segments $[x_i , x_{\tau(i)}+(0,1)], i=1,\cdots,m$ in $\R\times [1,2]$ where the $x_i$ are $m$ distinct points in $\R\times \{1\}$ such that $x_i$ is on the left of $x_{i+1}.$
The full diagram of $(p,\tau)$ is obtained by stacking the diagram of $\tau$ on top of the diagram of $p$ such that $x_i$ is the $i$th leaf of $p$.
Given symmetric forests $(q,\tau)\in\cSF(n,m), (p, \sigma )\in \cSF(m,l)$, let $l_i$ be the number of leaves of the $i$th tree of $p$, then we define the composition of morphisms as follows:
$$(p,\sigma)\circ (q,\tau) := ( \tau(p) \circ q , \sigma S( p , \tau ) ),$$
where $\tau(p)$ is the forest obtained from $p$ by permuting its trees such that the $i$th tree of $\tau(p)$ is the $\tau(i)$th tree of $p$ and $S(p,\tau)$ is the permutation corresponding to the diagram obtained from $\tau$ where the $i$th segment $[x_i , x_{\tau(i)} + (0,1)]$ is replaced by $l_{\tau(i)}$ parallel segments.
Thompson's group $V$ is isomorphic to the group of fractions of the category $\cSF$.
Hence, any element of $V$ is an equivalence class of a pair of symmetric \emph{trees}.
Consider $g=\frac{(t,\tau)}{(s,\sigma)} \in V$ and the standard dyadic partitions $(I_1,\cdots,I_n)$ and $(J_1,\cdots,J_n)$ of $[0,1]$ associated to the trees $s$ and $t$ respectively. The element $g$ acting on $[0,1]$ is the unique piecewise linear function with constant slope on each $I_k$ that maps $I_{\sigma^{-1}(i)}$ onto $J_{\tau^{-1}(i)}$ for any $1\leq i\leq n.$

Consider the cyclic group $\Z/m\Z$ as a subgroup of the symmetric group $S_m$ and the subcategory $\cAF\subset \cSF$ of \emph{affine} forests where $\cAF(n,m)=\cF(n,m)\times \Z/m\Z$.
The group of fractions of $\cAF$ is isomorphic to Thompson's group $T$.
We will often identify $\cF$ and $\cAF$ as subcategories of $\cSF$ giving embeddings at the group level $F\subset T\subset V$.

We say that a pair of symmetric trees $( (t , \tau ) , (s, \sigma) )$ is reduced if there are no pairs $( (t',\tau') , (s',\sigma'))$ such that $t'$ has strictly less leaves than $t$ and such that $\frac{(t,\tau)}{(s,\sigma)} = \frac{(t',\tau')}{(s',\sigma')}.$

Let $\Hilb$ be the category of complex Hilbert spaces with isometries for morphisms.
Given an isometry $R:\fH\to\fH\ot\fH$ we construct a functor
$\Phi=\Phi_R:\cF\to \Hilb$ such that $\Phi(n) = \fH^{\ot n}$ and $\Phi(f_{i,n})=\id^{\ot i-1}\ot R\ot \id^{\ot n-i}$ for $i=1,\cdots , n$.
Consider the quotient space 
$$\{ (t , \xi) : t \text{ tree} , \xi\in \Phi(\target(t)) \}/\sim \text{ generated by } (t,\xi)\sim (p\circ t, \Phi(p)\xi), \forall p\in\cF.$$
This quotient space has a pre-Hilbert structure given by $\langle (t,\xi) , (t,\eta)\rangle : = \langle \xi , \eta\rangle$ that we complete into a Hilbert space $\scrH$.
Note that $\scrH$ is the inductive limit of the system of Hilbert spaces $\fH_t:=\{(t,\xi) : \xi \in \Phi(\target(t)) \}$ for trees $t$ such that the embedding $\fH_t\to \fH_{p\circ t}$ is given by $\Phi(p)$.
We denote by $(t,\xi)$ or $\frac{t}{\xi}$ the equivalence class of $(t,\xi)$ inside $\scrH$ and identify $\fH$ and $\fH_t$ as subspaces of $\scrH$.
We have a unitary representation $\pi:F\to\cU(\scrH)$ given by the formula $\pi(\frac{t}{s}) \frac{s}{\xi}:=\frac{t}{\xi}$ that we extend to the group $V$ as follows:
$$\pi \left( \frac{(t,\tau)}{(s,\sigma)} \right) \frac{s}{\xi} := \frac{t}{  \theta(\tau^{-1}\sigma) \xi}, \text{ where } \theta(\kappa)(\eta_1\ot\cdots\ot \eta_n):=\eta_{\kappa^{-1}(1)}\ot\cdots\ot \eta_{\kappa^{-1}(n)}.$$
Note that if $\xi,\eta$ are in the small Hilbert space $\fH$ and $g=\frac{(t,\tau)}{(s,\sigma)} $, then 
\begin{equation}\label{equa:vacuum}
\langle \pi  \left( \frac{(t,\tau)}{(s,\sigma)} \right) \xi , \eta \rangle = \langle \theta(\sigma)\Phi(s) \xi , \theta(\tau) \Phi(t) \eta \rangle.\end{equation}

Consider an orthonormal basis $\{\xi_i : i\in I\}$ of the Hilbert space $\fH$.
The isometry $R$ can be thought of as a possibly infinite matrix with three indices $R_i^{j,k}$ such that $R \xi_i=\sum_{j,k} R_i^{j,k} \xi_j\ot \xi_k$.
We can reinterpret $\Phi(f)$ for a forest $f$ as a partition function.
Given a forest $f$ with $n$ roots and $m$ leaves, we define the set of states $\Omega(f)$ on $f$ as maps $\omega$ from the edges of $f$ to the set of indices $I$.
A \emph{vertex} of $f$ is a trivalent vertex and thus roots and leaves are not vertices.
If $\omega$ is a state on $f$ and $v$ a vertex, then we put $R^\omega_v$ the scalar equal to $R_{\omega(e_-)}^{\omega(e_l),\omega(e_r)}$ where $e_-$ is the edge with target $v$ and $e_l,e_r$ are the edges with source $v$ which goes to the left and right respectively.
Consider some multi-indices $\un i:=(i_1,\cdots,i_n)\in I^n$ and $\un j:=(j_1,\cdots,j_m)\in I^m$ and say that a state $\omega\in\Omega(f)$ is compatible with $(\un i, \un j)$ if $\omega(a_k)=i_k$ and $\omega(b_\ell)=j_\ell$ for all $1\leq k\leq n,1\leq \ell\leq m$ where $a_k$ is the edge with source the $k$th root of $f$ and $b_\ell$ is the edge with target the $\ell$th leaf of $f$.
We can now define the operator $\Phi(f)$ as follows:
$$\langle \Phi(f)\xi_{i_1}\ot\cdots\ot \xi_{i_n} , \xi_{j_1}\ot\cdots\ot\xi_{j_m} \rangle = \sum_{\begin{subarray}{c}\omega\in \Omega(f)\\  \text{ compatible with} (\un i, \un j) \end{subarray}} \prod_{v \text{ a vertex of } f} R_v^\omega$$
with the convention that a product (resp. a sum) over an empty set is equal to one (resp. zero). 
The infinite sum converges since the scalars $R_i^{j,k}$ are matrix coefficients of an isometry.
This formula can be proved by induction on the number of leaves and using the fact that any forest is the composition of some elementary forests $f_{i,n}$.
For example, if we have the following state $$\omega(f)=\fonefone$$ that we represent with the spin of an edge (i.e. the image by $\omega$ of this edge) next to it, then $$\prod_{v \text{ a vertex of } f} R_v^\omega = R^{l,m}_j R^{j,k}_i.$$

\section{Kazhdan's property (T)}\label{sec:T}
Recall that a countable discrete group $G$ has Kazhdan's property (T) if any unitary representation having an almost invariant vector has in fact a nonzero invariant vector \cite{Kazhdan67-T}.

If $u\in\cU(\fH)$ is a unitary and $\zeta\in\fH$ is a unit vector, then the map 
$$R:\fH\to \fH\ot\fH, \xi\mapsto u(\xi)\ot \zeta$$ 
is an isometry which provides us a unitary representation $\pi:V\to \cU(\scrH)$ as described in Section \ref{sec:def}.
We claim that if $ |\langle \zeta , u\zeta\rangle|  \neq 1$, then $\pi$ has no nonzero $[F,F]$-invariant vectors. 
Consider the following four trees 
$$a=f_3f_3f_1f_1 , b = f_4f_3f_2f_1 , c = f_1f_1 , d =f_2f_1$$ and let $t_n$ be the complete binary trees with $2^n$ leaves.
Put $g=\frac{a}{b}, h=\frac{c}{d},$ and $k:=ghg^{-1}h^{-1}$.
Note that 
$$k=\frac{a}{q} \text{ with } q= f_2f_3f_1f_1.$$
Define the element $k_n:=\frac{(a)_n\circ t_n}{(q)_n\circ t_n}$ where $(a)_n$ is the forests with $2^n$ roots and whose each tree is a copy of $a$.
Similarly define $g_n:=\frac{(a)_n\circ t_n}{(b)_n\circ t_n}$ and $h_n=\frac{(c)_n\circ t_n}{(d)_n\circ t_n}$ and observe that $k_n=g_nh_n(g_n)^{-1}(h_n)^{-1}$.
Therefore $k_n$ is in the commutator subgroup $[F,F]$.
Observe that
\begin{align*}
\langle \Phi(q) \xi ,\Phi(a) \eta \rangle & = \langle u^2 \xi \ot u\zeta \ot \zeta \ot u\zeta \ot \zeta , u^2 \eta \ot \zeta \ot u^2\zeta \ot \zeta\ot \zeta  \rangle\\
& = \langle \xi , \eta \rangle \langle u\zeta , \zeta \rangle^2 \langle \zeta , u^2\zeta\rangle, \ \forall \xi , \eta \in\fH 
\end{align*}
and thus $\Phi(a)^*\Phi(q) = C\cdot \id\in B(\fH)$ with $C =  \langle u\zeta , \zeta \rangle^2 \langle \zeta , u^2\zeta\rangle.$

If $\xi:=\ot_{i=1}^{2^n} \xi_i,\eta:=\ot_{i=1}^{2^n} \eta_i$ are elementary tensors of $\fH_{t_n}$, then 
$$\langle \pi(k_n) \xi , \eta \rangle  = \langle \Phi((q)_n)\xi,\Phi((a)_n)\eta\rangle
 = \prod_{i=1}^{2^n} \langle \Phi(q)\xi_i,\Phi(a)\eta_i\rangle 
 =C^{2^n} \langle \xi, \eta \rangle.$$
By linearity and density we obtain that $p_n\pi(k_n)p_n=C^{2^n} p_n$ where $p_n$ is the orthogonal projection onto $\fH_{t_n}$.
Assume that $\xi\in\scrH$ is a $[F,F]$-invariant unit vector and that $|\langle \zeta , u\zeta\rangle|<1$.
This implies that $| C |<1$.
By density, there exists $n$ and a unit vector $\xi'\in \fH_{t_n}$ such that $\Vert \xi - \xi'\Vert<1/4$.
We obtain that 
$$|C|^{2^n} = |\langle\pi(k_n)\xi',\xi'\rangle| \geq |\langle\pi(k_n)\xi,\xi\rangle|-|\langle\pi(k_n)(\xi'-\xi),\xi'\rangle|-|\langle\pi(k_n)\xi',(\xi'-\xi)\rangle|\geq \frac{1}{2},$$
for $n$ as large as we want which implies a contradiction since $|C|<1$ and proves the claim.

Consider $\fH:=\ell^2(\Z)$, the shift operator $u\in\cU(\ell^2(\Z))$, and $\zeta_m\in\ell^2(\Z)$ the characteristic function of $\{1,2,\cdots, h(m)\}$ divided by $\sqrt{h(m)}$ where $h(m)=2m8^m, m\geq 1$.
Let $(\pi(m),\scrH(m))$ be the associated sequence of unitary representations of $V$ and put $(\pi,\scrH):=(\oplus_m\pi(m), \oplus_m\scrH(m))$ their direct sum.
Note that $0<\langle \zeta_m , u\zeta_m\rangle<1$ for any $m\geq 1$ and thus the claim implies that $\pi$ does not have any nonzero $[F,F]$-invariant vectors.
Let $\eta_m$ be the elementary tensor of $\fH^{\ot 2^m}$ where each entry is equal to $\zeta_m$ that we identify with the fraction $\frac{t_m}{\eta_m}\in\fH_{t_m}$ view as an element of $\scrH(m)$.
Define $\xi_m$ to be the unit vector of $\scrH:=\oplus_n \scrH(n)$ which is equal to $\frac{t_m}{\eta_m}$ in the $m$th slot and zero elsewhere.
We claim that the sequence $(\xi_m,m\geq 1)$ is an almost $V$-invariant vector.
Fix $g\in V$ and note that for $m$ big enough there exists $s\in\cF(1,2^m), \kappa,\rho \in S_{2^m}$ such that $g=\frac{(s,\kappa)}{(t_m,\rho)}$.
Moreover, by increasing $m$, we can assume that there exists $p$ such that $p\circ s = t_{2m}$.
We obtain that
\begin{align*}
\langle \pi(g)\xi_m,\xi_m\rangle & =\langle \pi_m \left( \frac{(s,\kappa)}{(t_m,\rho)} \right) \frac{t_m}{\eta_m} , \frac{t_m}{\eta_m} \rangle \\
& =  \langle\frac{s}{\theta(\kappa^{-1}\rho)(\eta_m)},  \frac{t_m}{\eta_m} \rangle = \langle \frac{s}{\eta_m} , \frac{t_m}{ \eta_m} \rangle\\
& = \langle \Phi(p)\eta_m , \Phi(q)\eta_m \rangle  \text{ where } p\circ s =t_{2m} =  q\circ t_m \\
& = \prod_{i \text{ a leaf of } t_{2m} } \langle ( \Phi(p)\eta_m)_i , (\Phi(q)\eta_m)_i \rangle \text{ where } ( \Phi(p)\eta_m)_i \text{ is the $i$th tensor.}
\end{align*}
Since $p\circ s = t_{2m} = q\circ t_m$ we have that any branch of $p$ (resp. $q$) has length smaller or equal than $2m$ (resp. $m$).
This implies that any component of $\Phi(p)\eta_m$ and $\Phi(q)\eta_m$ is equal to $u^a\zeta_m$ for some $0\leq a\leq 2m$.
Since the map $a\in\N\mapsto \langle u^a\zeta_m , \zeta_m\rangle\in \R_+$ is decreasing we obtain 
$$\langle ( \Phi(p)\eta_m)_i , (\Phi(q)\eta_m)_i \rangle\geq \langle u^{2m}\zeta_m,\zeta_m\rangle=\frac{h(m)-2m}{h(m)}.$$
Therefore,
$$\langle \pi(g)\xi_m,\xi_m\rangle\geq \left(\frac{h(m)-2m}{h(m)}\right)^{2^{2m}}=(1-8^{-m})^{4^{m}}\longrightarrow_{m\to \infty} 1,$$
and thus $(\xi_m)_m$ is an almost $V$-invariant vector.
Since $\pi$ does not have any nonzero $[F,F]$-invariant vectors we obtain that any intermediate group between $[F,F]$ and $V$ does not have property (T).

\section{Haagerup property}\label{sec:H}
Recall that a countable discrete group $G$ has the Haagerup property if there exists a sequence $\varphi_n$ of positive definite functions which vanish at infinity on $G$ and such that $\lim_n\varphi_n(g)=1,\ \forall g\in G$ \cite{Akemann-Walter81}.

Consider the free group $\mathbb F_2$ freely generated by $a,b$ and let $\{\de_g:g\in \bF_2\}$ be its classical orthonormal basis.
Identify $\ell^2(\bF_2)^{\ot n}$ with $\ell^2(\bF_2^n)$ and $\de_{g_1}\ot\cdots\ot\de_{g_n}$ with $\de_{g_1,\cdots, g_n}$.
Set $\fH:=\ell^2(\bF_2)$ and define for $0\leq \al\leq 1$ the isometry 
$$R_\al:\ell^2(\bF_2) \to \ell^2(\bF_2\times \bF_2), \de_e\mapsto \al\de_{e,e} + \sqrt{1-\al^2} \de_{a,b}, \de_g\mapsto \de_{ag,bg}, \forall g\in\bF_2,g\neq e.$$
This defines a functor $\Phi_\al : \cF \to \Hilb$ and a unitary representation $\pi_\al:V\to \cU(\scrH_\al)$ as described in Section \ref{sec:def}. 
The associated infinite matrix $(R_g^{h,k}:=\langle R\de_g,\de_h\ot \de_k\rangle)_{g}^{h,k}$ is particularly simple since 
$$R_e^{e,e}=\al, R_e^{a,b}=\sqrt{1-\al^2},R_g^{ag,bg}=1, \forall e\neq g\in \bF_2$$
and zero elsewhere.

Consider the case when $\al=0$ and thus $R_g^{ag,bg}=1$ for any $g\in \bF_2$ and zero elsewhere.
Let $f$ be a forest with $n$ roots and $m$ leaves and observe that $\Phi_0(f)\de_{e,\cdots ,e}=\de_{P(f)}$ where $P(f)\in \bF_2^m$.
The $i$th component $P(f)_i$ is the word in $a,b$ written from right to left corresponding to the path from a root of $f$ to its $i$th leaf such that a left turn (resp. right turn) contributes in adding the letter '$a$' (resp. the letter '$b$').
For example, if $f=f_{1,4} \circ f_{3,3} \circ f_{1,2}$, then $ \Phi_0(f)\de_{e,e} = \de_{aa} \ot \de_{ba} \ot \de_{b} \ot \de_a \ot \de_b = \de_{aa , ba , b , a , b }$.
The next lemma proves that if $t$ is a tree then the set of words in the tuple $P(t)$ remembers completely $t$.

\begin{lemma}\label{lem:perm}
Consider two trees $s,t$ with $n$ leaves.
Assume that there exists a permutation $\sigma\in S_n$ acting on the leaves such that $\sigma(P(s)) = P(t)$.
Then $\sigma=\id$ and $s=t$.
\end{lemma}
\begin{proof}
We prove the lemma by induction on the number of leaves $n\geq 1$.
The result is immediate for $n=1$ and is also clear for $n=2$ since there is only one tree with two leaves and thus $P(s) = P(t) = (a,b)$. The permutation $\sigma$ is necessarily trivial.

Suppose the result is true for any $k$ between $1$ and $n$ and consider $s,t$ trees with $n+1$ leaves and a permutation $\sigma$ such that $\sigma(P(s)) = P(t)$.
Note that there exists trees $s_1,s_2,t_1,t_2$ such that $s = (s_1\bullet s_2)\circ f_1$ and $t=(t_1\bullet t_2)\circ f_1$ where $s_1\bullet s_2$ is the forest with two roots whose first tree is $s_1$ and second $s_2$.
Note that the word $P(s)_i$ finishes by the letter $a$ (resp. the letter $b$) if and only if $i$ is a leaf of $s_1$ (resp. a leaf of $s_2$).
We have the same characterization for the leaves of $t$ and thus necessarily $\sigma$ realizes a bijection from the leaves of $s_j$ onto the leaves of $t_j$ for $j=1,2$.
Observe that $P(s)_i = P(s_1)_ia$ for any leaves of $s_1$.
This implies that $\sigma(P(s_1)) = P(t_1)$.
Similarly we have that $\sigma(P(s_2)) = P(t_2)$ and thus by the induction hypothesis we have that $s_1=t_1, s_2=t_2$ and $\sigma$ is the identity on the leaves of $s_1$ and on the leaves of $s_2$ implying that $\sigma=\id$ and $s=t$.
\end{proof}

The lemma implies that $\pi_0$ contains the left regular representation of $V$.
Indeed, consider some symmetric trees $(t,\tau) , (s,\sigma)$ in $\cSF$ with the same number of leaves.
We have that
\begin{equation}\label{innerproduct}\langle \pi_0 \left( \frac{(t,\tau)}{(s,\sigma)} \right) \de_e , \de_e \rangle = \langle \theta(\sigma)\Phi(s)\de_e , \theta(\tau)\Phi(t)\de_e\rangle = \langle \de_{P(s)} , \de_{\sigma^{-1}\tau(P(t)) } \rangle.\end{equation}
This is nonzero (and then equal to one) if and only if $P(s)=\sigma^{-1}\tau(P(t))$.
In that case Lemma \ref{lem:perm} implies that $s=t$ and $\sigma^{-1}\tau=\id$ and thus $\frac{(t,\tau)}{(s,\sigma)}$ is the trivial group element.
We obtain that the cyclic representation generated by $\de_e$ for $\al=0$ is the left regular representation of $V$.
If $\al=1$, then the cyclic representation generated by $\de_e$ is the trivial one.
Indeed, if $t$ is a tree with $n$ leaves, then $\Phi(t)\de_e = \de_{e}\ot\cdots\ot\de_e$ with $n$ tensors and thus the coefficient \eqref{innerproduct} is always equal to one for any choice of $g=\frac{(t,\tau)}{(s,\sigma)}\in V.$
Our family $\pi_\al$ of representations of $V$ provides an interpolation between the trivial and the left regular representations.

Consider the family of positive definite maps $\vara(g):=\langle\pia(g)\de_e,\de_e\rangle$ with $0\leq \al<1$.
We will show that they vanish at infinity for $g\in T$ and tends to the identity when $\al$ tends to one.
Fix $0\leq \al<1$ and write $R$ instead of $R_\al$.
Section \ref{sec:def} tells us that given a tree $t$ with $n$ leaves we have the formula
$$\langle \Phi_\al(t)\de_e,\de_{\un j}\rangle = \sum_{\begin{subarray}{c}\omega\in\Omega(t)\\ \text{compatible with } (e,\un j) \end{subarray}}  \prod_{v \text{ a vertex of } t}R_v^\omega,$$
for any multi-index $\un j=(j_1,\cdots,j_n)$.
Fix $\omega\in\Omega(t)$ and assume that the $\omega$ coefficient of above is nonzero for a certain $\un j$.
Then there exists a maximal subrooted tree $z_\omega$ of $t$ such that the spin (i.e. the value of $\omega$) at each of its edges is the trivial group element $e\in\bF_2$.
Any vertex $v$ of the tree $z_\omega$ satisfies that $R_v^\omega=\al$.
If $f_\omega$ is the unique forest satisfying that $t=f_\omega\circ z_\omega$, then we have that any root $w$ of $f_\omega$ that is not a leaf of $f_\omega$ satisfies that $R_w^\omega=\sqrt{1-\al^2}$ since spins around it are necessarily $e,a,b$.
Then any other vertex $u$ of $f_\omega$ has its spins around it equal to $g,ga,gb$ for some $e\neq g\in \mathbb F_2$, and thus $R_u^\omega=1$.
We obtain that 
$$\prod_{v \text{ a vertex of } t}R_v^\omega=\al^{\target(z_\omega)-1} (1-\al^2)^{m(t,z_\omega)/2},$$ 
where $m(t,z_\omega)$ is the number of leaves of $z_\omega$ that are not leaves of $t$ (i.e. the number of nontrivial trees of $f_\omega$).
The spin of the edge with target the $\ell$th leaf of $t$ is the word in $a,b$ corresponding to the path in the forest $f_\omega$ starting at the root connected to $\ell$ and finishing at the leaf $\ell$.
Write $P(t,z_\omega)_\ell$ this word and $P(t,z_\omega)$ the corresponding $n$-tuple and note that by definition the multi-index $\un j$ is equal to $P(t,z_\omega)$.
Therefore given a multi-index $\un j$ there is at most one state $\omega\in\Omega(f)$ compatible with $(e,\un j)$ and having a nonzero coefficient $\prod_v R_v^\omega$. 
Moreover, $\un j$ has to be of the form $P(t,z)$ for some subrooted tree $z\leq t$.
Conversely, any subrooted tree $z$ of $t$ provides a unique state $\omega_z\in\Omega(t)$ that is compatible with $(e, P(t,z))$ defined inductively as follows: 
$$\omega_z(c) = \begin{cases} e & \text{ if $c$ is an edge of $z$};\\
a . \omega_z(d) & \text{ if $c$ goes to the left, its source is the target of an edge $d$ and } c\notin z;\\
b . \omega_z(d) & \text{ if $c$ goes to the right, its source is the target of an edge $d$ and } c\notin z.\\
\end{cases}
$$
We obtain the following formula
$$\Phi_\al(t) \de_e = \sum_{z\in E_t} \al^{\target(z)-1} (1-\al^2)^{m(t,z)/2} \de_{P(t,z)},$$
where $E_t$ is the set of subrooted trees (including the trivial subtree) of $t$.

Consider a pair of symmetric trees $( (t,\tau) , (s,\sigma) )$ and $g=\frac{(t,\tau)}{(s,\sigma)}\in V$.
We have that
\begin{equation}\label{equa:vara}\vara(g)= \sum_{z\in E_t, r\in E_s} \al^{\target(z)+\target(r)-2} (1- \al^2)^{(m(t,z)+m(s,r))/2} \langle \theta( \sigma ) \de_{ P(s,r) } , \theta( \tau) \de_{ P(t,z) } \rangle . \end{equation}
Fix a \emph{reduced} pair of \emph{affine} trees $( (t,\tau) , (s,\sigma) )$ and the group element $g=\frac{(t,\tau)}{ (s,\sigma)}$ that is in Thompson's group $T$ since our trees are affine.
We will show that all the terms in the sum \eqref{equa:vara} are equal to zero but one.

\begin{lemma}\label{lem:permtwo}
Consider some forests $p,q$ both of them having $m$ leaves and let $\sigma\in\Z/m\Z < S_m$ be a cyclic rotation.
If $\sigma(P(p)) = P(q)$, then $p$ and $q$ have the same number of roots $n$ and there exists $a\geq 0$ such that the $j$th tree of $p$ is equal to the $(j+a)$th modulo $n$ tree of $q$ for any  $1\leq j\leq n$.
\end{lemma}

\begin{proof}
Observe that if $f$ is a forest then the word $P(f)_i$ is a power of $a$ (resp. a power of $b$) if and only if $i$ corresponds to the first leaf (resp. the last leaf) of a tree of $f$.
Consider $p,q,\sigma$ as above and write $p_j$ and $q_k$ the $j$th and $k$th trees of $p$ and $q$ respectively.
Fix $j$ and note that the observation implies that there exists some natural numbers $a,b$ such that the first and the last leaves of $p_j$ are sent to the first leaf of $q_{j+a}$ and the last leaf of $q_{j+b}.$
If $a\neq b$ modulo the number of roots of $q$, then $P(p_j)$ would be equal to a tuple $P(f)$ of a forest $f$ having at least two trees and thus having at least two words that are a power of $b$ which is impossible.
Therefore, $\sigma$ realizes a bijection from the leaves of $p_j$ onto the leaves of $q_{j+a}$ for a certain $a$.
Since $\sigma$ is cyclic we obtain that the number $a$ does not depend on $j$.
We obtain that $P(p_j)=P(q_{j+a})$ for any $j$ which, by Lemma \ref{lem:perm}, implies that the trees $p_j$ and $q_{j+a}$ are equal.
\end{proof}

Assume that the $(z,r)$-term of the equality \eqref{equa:vara} is nonzero, then $P(s,r) = \sigma^{-1}\tau (P(t,z))$.
If $f(s,r), f(t,z)$ are the forest satisfying that $s = f(s,r)\circ r, t = f(t,z)\circ z$, then Lemma \ref{lem:permtwo} implies that there exists a cyclic permutation $\rho$ on the roots of $f(t,z)$ such that the $i$th tree of $f(s,r)$ is equal to the $\rho(i)$th tree of $f(t,z)$ and thus $s=\rho(f(t,z))\circ r$.
This implies that $g$ can be reduced as a fraction $\frac{(z,\tilde\tau)}{(r,\tilde \sigma)}$ for some permutations $\tilde\tau,\tilde\sigma$.
Since the pair $( (t,\tau) , (s,\sigma) )$ is already reduced we obtain that $z=t$ and $r=s$ and thus all the terms in the equality \eqref{equa:vara} are equal to zero except one.
Therefore, 
\begin{equation}\label{equa:varatwo}
\vara\left(\frac{(t,\tau)}{(s,\sigma)} \right)=\al^{2\target(t)-2} \text{ for any reduced pair of affine trees } ((t,\tau) , (s,\sigma)).\end{equation}
This implies that 
\begin{align*}
\lim_{g\to\infty}\vara(g) = 0 &  \text{ inside } T \text{ for any } 0\leq \al <1 ; \\
\lim_{\al\to 1} \vara(h) = 1 & \text{ for any } h\in T.
\end{align*}
Therefore, Thompson's group $T$ has the Haagerup property.

\begin{remark}\label{rem:Farley}
Farley constructed a proper cocycle $c:V\to H$ with values in a Hilbert space and showed in the proof of \cite[Theorem 2.4]{Farley03-H} that if $g\in V$ is described by a reduced pair of symmetric trees with $n$ leaves, then $\Vert c(g)\Vert^2 = 2n-2.$
This provides a family of positive definite functions $\phi_\beta(g):= \exp(-\beta \Vert c(g)\Vert^2)=\exp(-\beta(2n-2)), \beta\geq 0$ by Schoenberg's theorem.
Note that if $g\in T$, then Formula \eqref{equa:varatwo} implies $\vara(g) = \phi_{\exp(-\beta)}(g)$ for any $\al>0$.
However, this equality is no longer true for certain elements of $V$.
Consider $g=\frac{(t,\id)}{(t,(13))}\in V$ where $t=f_3f_1f_1$ is the full binary tree with four leaves. We have that $$\Phi_\al(t)\de_e = \al^3 \de_{e,e,e,e} + \al^2\sqrt{1-\al^2} ( \de_{e,e,a,b} + \de_{a,b,e,e} ) + \al(1-\al^2)\de_{a,b,a,b} + \sqrt{1-\al^2}\de_{aa,ba,ab, bb}.$$ Then $\vara(g) = \al^6 + \al^2(1-\al^2)^2\neq \al^6= f_{\exp(-\beta)}(g) $.
\end{remark}

We proved that $T$ has the Haagerup property by using the net of maps $\vara, 0<\al<1.$
One could hope to extend our proof using the same approximation of the identity for the larger group $V$. 
Unfortunately, the maps $\vara$ with $0<\al<1$ are no longer vanishing at infinity if we consider them as functions on $V$.
Indeed, consider the sequence of trees $x_n$ such that $x_2=f_1$ is the tree with two leaves and $x_{n+1}=f_1 x_n$ for $n\geq 2.$
Note that $x_n$ is a tree with $n$ leaves.
Let $s_n:=(x_n\bullet x_n)\circ f_1$ be the tree equal to the composition of $f_1$ with a copy of $x_n$ attached to each leaf of $f_1$.
Define the permutation $\sigma_n \in S_{2n}$ that is an involution and such that $\sigma_n(2i+1)=2i+1 + n$ and $\sigma_n(2j) =2j$ for any $1 \leq 2i+1 , 2 j\leq n$.
Hence, $\sigma_n$ sends any odd leaf of the first copy of $x_n$ to the same leaf in the second copy of $x_n$ and lets invariant the others.
We set $g_n:=\frac{(s_n,\sigma_n)}{(s_n,\id)}$ and note that this fraction is reduced.
If we consider $\vara(g)$ we observe that the term corresponding to $z=r=f_1$ in the formula \eqref{equa:vara} is nonzero and is equal to $\al^2 (1-\al^2)^2$.
Since all the terms of $\vara(g_n)$ are positives we obtain that $\vara(g_n)\geq \al^2 (1-\al^2)^2 >0$ for any $n\geq 2$ and thus $\vara(g)$ does not tend to zero when $g$ tends to infinity in $V$.

\end{document}